\documentclass[12pt]{amsart}
\usepackage{amsmath}
\usepackage{amsxtra}
\usepackage{amscd}
\usepackage{amsthm}
\usepackage{amsfonts}
\usepackage{amssymb}
\usepackage {pstricks}
\usepackage{pstricks,pst-node}
\usepackage{tikz}

\newtheorem{thm}{Theorem}[section]

\newtheorem{lem}[thm]{Lemma}

\newtheorem{cor}[thm]{Corollary}
\newtheorem{prop}[thm]{Proposition}

\theoremstyle{definition}
\newtheorem{definition}[thm]{Definition}
\newtheorem{example}[thm]{Example}

\theoremstyle{remark}
\newtheorem{rem}[thm]{Remark}

\numberwithin{equation}{section}
\thanks
{The present work was initiated during a visit by the first author to University of Sydney in October 2012 and he gratefully acknowledges the support and hospitality extended to him there.  He was also supported by the Danish National Research Foundation center of Excellence, Center for Quantum Geometry of Moduli Spaces (QGM).\\
The second and third authors were supported by the Australian Research Council, through Discovery Projects DP120103432 (Quantised algebras, supersymmetry and invariant theory) and DP0772870 (Invariant theory, cellularity and geometry).}
\begin{document}
\normalfont
\newcommand{\thmref}[1]{Theorem~\ref{#1}}
\newcommand{\secref}[1]{Section~\ref{#1}}
\newcommand{\lemref}[1]{Lemma~\ref{#1}}
\newcommand{\propref}[1]{Proposition~\ref{#1}}
\newcommand{\corref}[1]{Corollary~\ref{#1}}
\newcommand{\remref}[1]{Remark~\ref{#1}}
\newcommand{\eqnref}[1]{(\ref{#1})}
\newcommand{\exref}[1]{Example~\ref{#1}}

\newcommand{\nc}{\newcommand}

\nc{\on}{\operatorname}

\nc{\Z}{{\mathbb Z}}
\nc{\bQ}{{\mathbb Q}}
\nc{\C}{{\mathbb C}}
\nc{\R}{{\mathbb R}}
\nc{\bbP}{{\mathbb P}}
\nc{\bF}{{\mathbb F}}

\nc{\boldD}{{\mathbb D}}
\nc{\oo}{{\mf O}}
\nc{\N}{{\mathbb N}}
\nc{\bib}{\bibitem}
\nc{\pa}{\partial}
\nc{\F}{{\mf F}}
\nc{\CA}{{\mathcal A}}
\nc{\cD}{{\mathcal D}}
\nc{\CE}{{\mathcal E}}
\nc{\CP}{{\mathcal P}}
\nc{\CO}{{\mathcal O}}
\nc{\CK}{{\mathcal K}}
\nc{\CN}{{\mathcal N}}
\nc{\CU}{{\mathcal U}}
\nc{\Res}{\text{Res}}
\nc{\Ind}{\text{Ind}}
\nc{\Ker}{\text{Ker}}
\nc{\id}{\text{id}}
\nc{\ot}{\otimes}

\nc{\be}{\begin{equation}}
\nc{\ee}{\end{equation}}

\nc{\rarr}{\rightarrow}
\nc{\larr}{\longrightarrow}
\nc{\al}{\alpha}
\nc{\ri}{\rangle}
\nc{\lef}{\langle}

\nc{\W}{{\mc W}}
\nc{\gam}{\ol{\gamma}}
\nc{\Q}{\ol{Q}}
\nc{\q}{\widetilde{Q}}
\nc{\la}{\lambda}
\nc{\ep}{\epsilon}
\nc{\ve}{\varepsilon}

\nc{\g}{\mf g}
\nc{\h}{\mf h}
\nc{\n}{\mf n}
\nc{\bb}{\mf b}
\nc{\G}{{\mf g}}

\nc{\D}{\mc D}
\nc{\cE}{\mc E}
\nc{\CF}{\mc F}
\nc{\CC}{\mc C}
\nc{\CH}{\mc H}
\nc{\CL}{\mc L}
\nc{\CT}{\mc T}
\nc{\CI}{\mc I}
\nc{\CR}{\mc R}

\nc{\UK}{{\mc U}_{\CA_q}}

\nc{\CS}{\mc S}

\nc{\CB}{\mc B}

\nc{\Li}{{\mc L}}
\nc{\La}{\Lambda}
\nc{\is}{{\mathbf i}}
\nc{\V}{\mf V}
\nc{\bi}{\bibitem}
\nc{\NS}{\mf N}
\nc{\dt}{\mathord{\hbox{${\frac{d}{d t}}$}}}
\nc{\E}{\mc E}
\nc{\ba}{\tilde{\pa}}
\nc{\half}{\frac{1}{2}}

\def\smapdown#1{\big\downarrow\rlap{$\vcenter{\hbox{$\scriptstyle#1$}}$}}

\nc{\mc}{\mathcal}
\nc{\ov}{\overline}
\nc{\mf}{\mathfrak}
\nc{\ol}{\fracline}
\nc{\el}{\ell}
\nc{\etabf}{{\bf \eta}}
\nc{\zetabf}{{\bf
\zeta}}\nc{\x}{{\bf x}}
\nc{\xibf}{{\bf \xi}} \nc{\y}{{\bf y}}
\nc{\WW}{\mc W}
\nc{\SW}{\mc S \mc W}
\nc{\sd}{\mc S \mc D}
\nc{\hsd}{\widehat{\mc S\mc D}}
\nc{\parth}{\partial_{\theta}}
\nc{\cwo}{\C[w]^{(1)}}
\nc{\cwe}{\C[w]^{(0)}} \nc{\wt}{\widetilde}
\nc{\gl}{\mf gl}
\nc{\K}{\mf k}

\newcommand{\U}{{\rm{U}}}
\newcommand{\TL}{{\rm{TL}}}
\newcommand{\Aut}{{\rm{Aut}}}
\newcommand{\ch}{{\rm{char}}}
\newcommand{\End}{{\rm{End}}}
\newcommand{\Hom}{{\rm{Hom}}}
\newcommand{\Mod}{{\rm{Mod}}}
\newcommand{\Rad}{{\rm{Rad}}}
\newcommand{\Ext}{{\rm{Ext}}}
\newcommand{\Tor}{{\rm{Tor}}}
\newcommand{\Lie}{{\rm{Lie}}}
\newcommand{\Uq}{{{\mathcal U}_v}}
\newcommand{\GL}{{\rm{GL}}}
\newcommand{\Sym}{{\rm{Sym}}}
\newcommand{\rk}{{\rm{rk}}}
\newcommand{\ord}{{\rm{ord}}}
\newcommand{\tr}{{\rm{tr}}}
\newcommand{\Rea}{{\rm{Re}}}
\newcommand{\rank}{{\rm{rank}}}
\newcommand{\im}{{\rm{Im}}}

\advance\headheight by 2pt

\nc{\fb}{{\mathfrak b}} \nc{\fg}{{\mathfrak g}}
\nc{\tA}{\widetilde{A}}
\nc{\fh}{{\mathfrak h}}  \nc{\fk}{{\mathfrak k}}

\nc{\fl}{{\mathfrak l}} \nc{\fn}{{\mathfrak n}}

\nc{\fp}{{\mathfrak p}} \nc{\fu}{u}


\nc{\fS}{{\Sym}}

\nc{\fsl}{{\mathfrak {sl}}} \nc{\fsp}{{\mathfrak {sp}}}
\nc{\fso}{{\mathfrak {so}}} \nc{\fgl}{{\mathfrak {gl}}}

\nc{\A}{\mc A} \nc{\cF}{{\mathcal F}}

\nc{\cA}{{\mathcal A}} \nc{\cP}{{\mathcal P}} \nc{\cC}{{\mathcal C}}
\nc{\cU}{{\mathcal U}} \nc{\cB}{{\mathcal B}}

\def\lr{{\longrightarrow}}
\def\inv{{^{-1}}}

\def\xl{{\hbox{\lower 2pt\hbox{$\scriptstyle \mathfrak L$}}}}

\nc{\bX}{{\mathbf X}} \nc{\bx}{{\mathbf x}} \nc{\bd}{{\mathbf d}}
\nc{\bdim}{{\mathbf dim}} \nc{\bm}{{\mathbf m}}

\title[Cellularity of quantum endomorphism algebras]{Cellularity of certain quantum endomorphism algebras.}

\author{H.H. Andersen, G.I. Lehrer and R.B. Zhang}
\address{QGM,
Det Naturvidenskabelige Fakultet,
Aarhus Universitet, Ny Munkegade Bygning 1530,
8000 Aarhus C, Denmark}
\address{School of Mathematics and Statistics,
University of Sydney, N.S.W. 2006, Australia}
\email{mathha@imf.au.dk, gustav.lehrer@sydney.edu.au, ruibin.zhang@sydney.edu.au}

\begin{abstract} Let $\tA=\Z[q^{\pm \frac{1}{2}}][([d]!)\inv]$ and let
$\Delta_{\tA}(d)$ be an integral form of the Weyl module
of highest weight $d \in \N$ of the quantised enveloping algebra $\U_{\tA}$ of $\fsl_2$.
We exhibit for all positive integers $r$ an explicit cellular structure for $\End_{\U_{\tA}}(\Delta_{\tA}(d)^{\ot r})$. When $\zeta$ is a root of
unity of order bigger than $d$ we consider
the specialisation $\Delta_{\zeta}(d)^{\ot r}$ at $q\mapsto \zeta$
of $\Delta_{\tA}(d)^{\ot r}$. We prove one general result which gives sufficient conditions for the commutativity of
specialisation with the taking of endomorphism algebras, and another which relates the multiplicities of indecomposable
summands to the dimensions of simple modules for an endomorphism algebra.
Our cellularity result then allows us to prove that knowledge of the dimensions of the simple modules of the specialised cellular 
algebra above is equivalent to knowledge of the weight multiplicities of the tilting modules for $\U_{\zeta}(\fsl_2)$.
In the final section we independently determine the weight multiplicities of indecomposable tilting modules for $U_\zeta(\fsl_2)$ and the  decomposition numbers of the endomorphism algebras. We indicate how our earlier results imply that either one of these sets of numbers determines the other.
\end{abstract}
\maketitle


\section{Introduction.}

\subsection{Notation}
Let $A$ be the ring $\Z[q^{\pm \frac{1}{2}}]$ where $q$ is an indeterminate,
and let $U_A$ be the Lusztig $A$-form \cite{L1,L2,L3} of the quantised enveloping algebra 
$\U_q(\fsl_2)$ \cite{D,J,CP},
which has basis consisting of `divided powers' of the generators of $\fsl_2$.
Let $\Delta_A(d)$ be the simple
$\U_A$-module with highest weight $d \in \N$. This has dimension $d+1$ and quantum dimension
equal to the quantum number $[d+1]$, where for any integer $n$, 
$[n]=[n]_q:=\frac{q^n-q^{-n}}{q-q\inv}$.

For any commutative  $A$-algebra $\tA$, we write $\U_{\tA}:=\tA\otimes_A \U_A$,
and similarly for $\Delta_{\tA}(d)$, etc.
For any positive integer $r$, let $E_r(d,\tA):=\End_{\U_{\tA}}(\Delta_{\tA}(d)^{\ot r})$.

Let $s_1,\dots,s_{N-1}$ be the standard Coxeter generators of $\Sym_{N}$. For $w\in \Sym_{N}$,
write $\ell(w)$ for its length as a word in the generators $s_i$, and
define the left set $L(w):=\{i\mid \ell(s_iw)<\ell(w)\}$; the
right set $R(w)$ is defined similarly.

\subsection{The main result} Let $K=\bQ(q^{\frac{1}{2}})$ be the
field of fractions of $A$.
Writing $B_r$ for the $r$-string braid group ($r$ a positive integer),
it is known that there is an
action of $B_r$ on $\Delta_{A}(d)^{\ot r}$, in which the standard generators of the
braid group act on successive tensor factors via the $R$-matrix $\check R$. This is
evident over $K$, and
from \cite{LZ1,LZ2}, and \cite{ALZ} or \cite{TA} (using \cite{KR}) in the above integral form. This action respects the
$\U_{\tA}$-action on the tensor space, and so there is a homomorphism
\be\label{eq:eta}
\eta: \tA B_r\lr \End_{\U_{\tA}}(\Delta_{\tA}(d)^{\ot r})=E_r(d,\tA).
\ee
We define $A$ using $q^{\frac{1}{2}}$ instead of $q$ because then with the usual definitions of $\U_q$, the $R$-matrix 
is defined over $A$ with respect to a basis of weight vectors.

In \cite{LZ1} it was shown that when $\tA=K$, $\eta$ is surjective. This provides a 
means of studying the relevant endomorphism algebras. When $d = 2$ this surjectivity was proved in \cite{TA} for most $\tilde A$. We haven't been able to establish this result for $d > 2$. However, inspired in part by the methods used in loc. cit. we show in this paper that the endomorphism algebras have a nice cellular structure, even though the $R$-matrix generators satisfy a polynomial equation of degree $d+1$.

We shall work with the Temperley-Lieb algebra $\TL_N(\tA)$, which has generators
$f_i$, $i=1,\dots,N-1$ and relations $f_if_{i\pm 1}f_i=f_i$ and $f_i^2=(q+ q\inv)f_i$.
This has an $\tA$-basis consisting of planar diagrams, as explained in \cite[\S 1]{GL96} (see also \cite{GL03,GL04});
these are in $1-1$ correspondance with the set of fully commutative elements of $\Sym_N$, see \cite{FG}.

We shall prove here that 
\begin{thm}\label{thm:cell} Let $d\geq 1$ be an integer.
For any $\tA$ such that $[d]!$ is invertible in 
$\tA$, the algebra $E_r(d,\tA)$ is isomorphic to a cellular subalgebra of $\TL_{rd}(\tA)$.
In particular, it has an $\tA$-basis labelled by planar diagrams
$D\in\TL_{rd}(\tA)$ such that $L(D),R(D)\subseteq \{d,2d,\dots,(r-1)d\}$, where the 
left and right sets $L(D)$ and $R(D)$ are as in Definition \ref{def:lrsets} below.
\end{thm}
Note that the planar diagrams are labelled by the set $\Sym_{rd}^c$ of fully commutative elements in
$\Sym_{rd}$; the requirement in the theorem is equivalent to
taking those $w\in\Sym_{rd}^c$ such that $L(w),R(w)\subseteq \{d,2d,\dots,(r-1)d\}$
(cf. \cite{FG}).

We shall give further details of the cellular structure below, both in terms of diagrams,
and in terms of pairs of standard tableaux.

\section{The case $d=1$.} 
\subsection{Temperley-Lieb action}
It is known (cf., e.g. \cite[\S 3.4]{LZ2}) that in this case,
the $R$-matrix acts on $\Delta_K(1)^{\ot 2}$ with eigenvalues $q^{\frac{1}{2}}$ and 
$-q^{\frac{3}{2}}$.
If we adjust the map $\eta$ of \eqref{eq:eta} by sending the generators to 
$T_i:=q^{\frac{1}{2}}R_i$,
where $R_i$ is the relevant $R$-matrix, then $\eta$ factors through the algebra
$H_r(A):=AB_r/\lef(T_i+q\inv)(T_i-q)\ri$, which is well known to be the Hecke algebra,
and has $A$-basis $\{T_w\mid w\in\Sym_r\}$. We therefore have, after tensoring with $\tA$,

\be\label{eq:mu}
\mu: H_r(\tA)\lr \End_{\U_{\tA}}(\Delta_{\tA}(1)^{\ot r})=E_r(1,\tA).
\ee

Moreover it is a special case of the main result of \cite{DPS99} (see also \cite{ALZ})
that $\mu$ is surjective for any choice of $\tA$, 
even when $\tA$ is taken to be $A$. Further, the arguments in \cite[Th. 3.5]{LZ2} generalised to the integral case show that the kernel
of $\mu$ is the ideal generated by the element $a_3:=\sum_{w\in \Sym_3}(-q)^{-\ell(w)}T_w$.
It follows that for any $\tA$, we have an isomorphism

\be\label{eq:nu}
\eta: H_r(\tA)/\lef a_3\ri\cong TL_r(\tA)\overset{\sim}{\lr}
\End_{\U_{\tA}}(\Delta_{\tA}(1)^{\ot r})=E_r(1,\tA),
\ee
where $TL_r(\tA):=H_r(\tA)/\lef a_3\ri$ is the $r$-string Temperley-Lieb algebra.
The generator $f_i$ acts as $q-T_i$ on $\Delta_{\tA}(1)^{\ot r}$. It is easily
shown that $f_i^2=(q+q\inv)f_i$, and that the other Temperley-Lieb relations are satisfied.

\subsection{Projection to $\Delta_{\tA}(d)$} Now it is elementary that
\be\label{eq:dd}
\Delta_K(1)^{\ot d}\cong \Delta_K(d)\oplus \Delta',
\ee
where $\Delta'$ is the direct sum of simple modules $\Delta_K(i)$ with $i<d$.
We therefore have a canonical projection $p_d:\Delta_K(1)^{\ot d}\lr\Delta_K(d)$,
which may be considered an element of $E_r(d,K)=\End_{\U_K}(\Delta_K(1)^{\ot d})$.

\begin{lem}\label{lem:pd} The projection $p_d$ is the image under $\mu$ (see \eqref{eq:mu})
of the element $e_d:=P_d(q)\inv\sum_{w\in\Sym_d}q^{\ell(w)}T_w \in H_d(\wt A)$, where
$P_d(q)=q^{\frac{d(d-1)}{2}}[d]!$.
\end{lem}
\begin{proof}
We begin by showing that for $i=1,\dots,d-1$, 
\be\label{eq:pd}
T_ip_d=p_dT_i=qp_d
\ee
as endomorphisms of $\Delta_K(1)^{\ot d}$.

By symmetry, it suffices to prove \eqref{eq:pd} for $i=1$. Now
$$
\begin{aligned}
\Delta_K(1)^{\ot d}=&\Delta_K(1)\ot\Delta_K(1)\ot\Delta_K(1)^{\ot (d-2)}\\
\cong& (\Delta_K(0)\oplus \Delta_K(2))\ot \Delta_K(1)^{\ot (d-2)}\\
\cong& (\Delta_K(0)\ot \Delta_K(1)^{\ot (d-2)})\oplus (\Delta_K(2)\ot \Delta_K(1)^{\ot (d-2)})\\
\end{aligned}
$$

But $p_d$ acts as zero on the first summand (since the highest occurring weight is $d-2$) and
$T_1$ acts as $q$ on the second summand. This proves the relation \eqref{eq:pd}.
Now since $f_i=\mu(q-T_i)$, this shows that $p_d$ is the ``Jones idempotent''
of $TL_d(K)$, defined by the relations $f_ip_d=p_df_i=0$ for all $i$. 

It follows that if $p_d'$ is the unique idempotent in $H_d(K)$ corresponding to the algebra
homomorphism $T_w\mapsto q^{\ell(w)}$, then $p_d=\mu(p_d')$. But this idempotent
is precisely the element $e_d$ in the statement.
\end{proof}

The next statement is immediate.

\begin{cor}\label{cor:pd}
Let $\tA=A[[d]!\inv]$. Then 
\be\label{eq:drd}
\Delta_{\tA}(1)^{\ot rd}\cong \Delta_{\tA}(d)^{\ot r}\oplus \Gamma,
\ee
where $\Gamma$ is a $\U_{\tA}$-submodule, and the corresponding
projection $p\in\End_{rd}(1,\tA)$ such that $p(\Delta_{\tA}(1)^{\ot rd})
=\Delta_{\tA}(d)^{\ot r}$ is given by $p=p_d^{\ot r}$ where we now consider 
$p_d$ as an element of $E_r(d, \tA) \subset E_r(d, K)$.
\end{cor}

\section{Endomorphisms of $\Delta_{\tA}(d)^{\ot r}$.}
\subsection{Identification of $E_r(d,\tA)$}
Throughout this section we take $\tA$ to be $\tA=A[[d]!\inv]$.
Recall that $E_r(d,\tA)=\End_{\U_{\tA}}(\Delta_{\tA}(d)^{\ot r})$.
We are now in a position to identify $E_r(d,\tA)$ on the nose, as a subalgebra
of $\TL_{rd}(\tA)\cong \End_{\U_{\tA}}(\Delta_{\tA}(1)^{\ot rd})$.
This will lead to the identification of the cellular structure on
$E_r(d,\tA)$.

\begin{prop}\label{prop:ed}
There is an
isomorphism $E_r(d,\tA)\overset{\sim}{\lr}p\TL_{rd}(\tA)p$, where $p$ is the 
idempotent $p=p_d^{\ot r}$ of $\TL_{rd}(\tA)$ described above.
\end{prop}
\begin{proof}
For any endomorphism $\alpha\in E_r(d,\tA)$ we obtain an endomorphism $\wt\alpha$ of
$\Delta_{\tA}(1)^{\ot rd}$ by extending $\alpha$ by zero, using the decomposition
\eqref{eq:drd}, that is, by defining $\wt \alpha$ to be zero on $\Gamma$.
The map $\alpha\mapsto\wt\alpha$ is an inclusion $E_r(d,\tA)\hookrightarrow E_{rd}(1,\tA)$,
and its image is clearly the space of endomorphisms
$\beta\in E_{rd}(1,\tA)$ such that $\ker(\beta)\supseteq \Gamma$ and 
$\im(\beta)\subset \Delta_{\tA}(d)^{\ot r}$ (as in the decomposition \eqref{eq:drd}).
This image is $p\TL_{rd}(\tA)p$.
\end{proof}

\subsection{Temperley-Lieb diagrams}\label{ss:tl}
The key step in proving cellularity is the identification of a 
certain $\tA$-basis of $p\TL_{rd}(\tA)p$. This will be done in terms of certain diagrams.
The Temperley-Lieb algebra $\TL_{rd}(\tA)$ has $\tA$ basis
consisting of planar diagrams from $rd$ to $rd$, in the language of  \cite{GL98}.
These diagrams are in bijection with the set $\Sym_{rd}^c$ of fully commutative elements \cite{FG}
of $Sym_{rd}$, which in turn is in bijection with those elements of $\Sym_{rd}$ which correspond,
under the Robinson-Schensted correspondence, to pairs of standard tableaux with two rows.

\begin{center}
\begin{tikzpicture}[scale=1.5]

\foreach \x in {1,2,3,4,5,6}
\filldraw(\x,0) circle (0.05cm);
\foreach \x in {1,2,3,4,5,6}
\filldraw(\x,2) circle (0.05cm);
\draw (1,0)--(5,2);
\draw (4,0)--(6,2);

\draw node[above] at (1,2){1};\draw node[above] at (2,2){2};\draw node[above] at (3,2){3};
\draw node[above] at (4,2){4};\draw node[above] at (5,2){5};\draw node[above] at (6,2){6};

\draw node[below] at (1,0){1};\draw node[below] at (2,0){2};\draw node[below] at (3,0){3};
\draw node[below] at (4,0){4};\draw node[below] at (5,0){5};\draw node[below] at (6,0){6};

\draw(1,2) .. controls (2,1.2) and (3,1.2) .. (4,2);
\draw(2,2) .. controls (2.2,1.7) and (2.8,1.7) .. (3,2);

\draw(2,0) .. controls (2.2,0.3) and (2.8,0.3) .. (3,0);
\draw(5,0) .. controls (5.2,0.3) and (5.8,0.3) .. (6,0);


\end{tikzpicture}
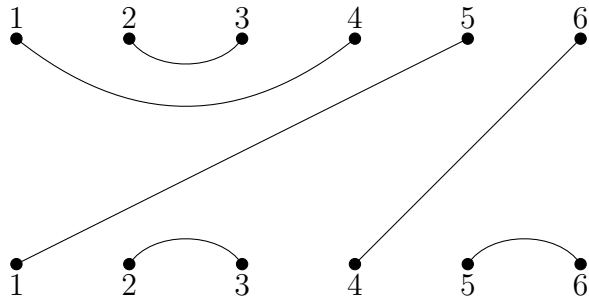
\centerline{Figure 1}
\end{center}
\label{fig2}

We shall describe now how to obtain a pair $(S(D),R(D))$ of standard tableaux directly from 
a planar diagram $D$. We use the planar diagram from $6$ to $6$
in Figure 1 to illustrate the description.

Each planar diagram from $N$ to $N$ consists of a set
of $N$ non-intersecting arcs. These may be through-arcs, joining an upper node to a lower node,
or upper (top to top) or lower (bottom to bottom). The latter two are referred to as horizontal arcs.
The diagrams are multiplied in the usual way, by concatenation, with each closed circle being
replaced by $[2]=q+q\inv$. The generator $f_i$ corresponds to the diagram in Figure 2. Note that 
if there are $t$ through arcs, then there are equally many top arcs and bottom arcs, and if this 
number is $k$, then $t+2k=N$.

\begin{center}
\begin{tikzpicture}[scale=1.5]

\foreach \x in {1,3,4,5,6,8}
\filldraw(\x,0) circle (0.05cm);
\foreach \x in {1,3,4,5,6,8}
\filldraw(\x,2) circle (0.05cm);
\draw (1,0)--(1,2);
\draw (3,0)--(3,2);
\draw (6,0)--(6,2);
\draw (8,0)--(8,2);

\draw node[above] at (1,2){1};\draw node[above] at (4,2){i};\draw node[above] at (5,2){i+1};
\draw node[above] at (8,2){N};

\draw node[below] at (1,0){1};\draw node[below] at (4,0){i};\draw node[below] at (5,0){i+1};
\draw node[below] at (8,0){N};

\draw(4,2) .. controls (4.2,1.7) and (4.8,1.7) .. (5,2);

\draw(4,0) .. controls (4.2,0.3) and (4.8,0.3) .. (5,0);

\draw(2,2) node {\large{$\cdots$}};
\draw(7,2) node {\large{$\cdots$}};
\draw(2,0) node {\large{$\cdots$}};
\draw(7,0) node {\large{$\cdots$}};

\end{tikzpicture}
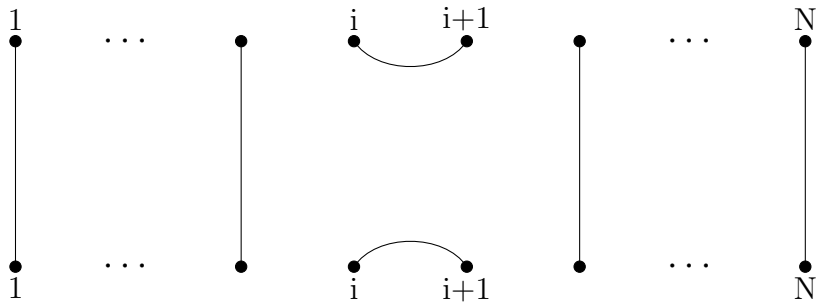
\centerline{Figure 2}
\end{center}
\label{fig3}

Now to each such planar diagram $D$, we associate an ordered pair $(S(D),T(D))$ of standard 
tableaux
with two rows, as follows. Let $i_1,\dots,i_k$ be the right nodes of the upper  arcs
written in ascending order. Then $S(D)$ has second row $i_1,\dots,i_k$, and first row
the complement of $\{i_1,\dots,i_k\}$, written in ascending order. Note that the first 
row has $t+k\geq k$ elements. The tableau $T(D)$ is defined similarly, using the sequence
$j_1,\dots,j_k$ of right ends of the lower arcs. Note that both $S(D)$ and $T(D)$ correspond to
the partition $(t+k,k)$, and hence the diagram corresponds via the Robinson-Schensted
correspondence to an element $w(D)\in\Sym_N$, which is fully commutative.

Say that a horizontal arc is {\em small} if its vertices are $i,i+1$ for some $i$. 

\begin{definition}\label{def:lrsets}The {\em left set}
$L(D)$ of a planar diagram $D$ is the set of left vertices of the small upper arcs of $D$.
Similarly, the
right set $R(D)$ is the set of left vertices of the small lower arcs of $D$. 
\end{definition}

Note that in the notation from Section 1.1 we have
$L(D)= L(w(D))$, and similarly $R(D)= R(w(D))$. 

For the diagram $D$ in Figure 1, $L(D)=\{2\}$, while $R(D)=\{2,5\}$. The tableaux $S(D)$ and $T(D)$
are given by
$$
S(D)=
\begin{tabular}{cccc}  
1&2&5&6\\
3&4&&\\
\end{tabular}
,\;\;\;
T(D)=
\begin{tabular}{cccc}  
1&2&4&5\\
3&6&&\\
\end{tabular} 
$$

Note that if $\cD(S):=
\{i\mid i+1\text{ is in a lower row than }i\}$ is the descent set of a standard tableau
$S$, then $L(D)=\cD(S(D))$ and $R(D)=\cD(T(D))$.

\section{Proof of the main theorem.} 

In this section we prove Theorem \ref{thm:cell}, and give some of its consequences. We keep the convention
$\tA = A[([d]!)\inv]$ from Section 3.

\subsection{A key Lemma} We begin by proving the following key result.

\begin{lem}\label{lem:basis} The $\tA$-algebra
$p\TL_{dr}(\tA)p$ has $\tA$-basis given by the set of elements $pDp$ where $D$ is a diagram in 
$\TL_{dr}(\tA)$ such that $L(D)\cup R(D)\subseteq\{d,2d,\dots,(r-1)d\}$.
\end{lem}
\begin{proof}
The $\tA$-algebra $E_r(d,\tA)\cong
p\TL_{rd}(\tA)p$ is evidently spanned by the elements $pDp$ where $D$ ranges over
the planar diagrams $:rd\to rd$. But for $i=1,\dots,d-1$, we have seen that 
$p_df_i=f_ip_d=0$. It follows that $pDp=0$ unless $L(D)$ and $R(D)$ are both contained
in $\{d,2d,\dots,(r-1)d\}$. Let $\CB(d,r)$ be the set of planar diagrams satisfying these 
conditions. By the above remarks, it will suffice to show that 
\be\label{eq:li}
\{pDp\mid D\in\CB(d,r)\}\text{ is linearly independent.} 
\ee

To prove \eqref{eq:li} it suffices to work over the field $K$; in particular
we are reduced to showing that 

\be\label{eq:dims}
|\CB(d,r)|=\dim_K\left(\End_{\U_K}(\Delta_K(d)^{\ot r}\right).
\ee

We shall prove \eqref{eq:dims} essentially by showing that both sides of \eqref{eq:dims}
satisfy the same recurrence. Let us begin with the left side.

Observe that if a diagram $D\in \CB(d,r)$ has $t$ through arcs, it may be thought of as a 
pair of diagrams $D_1,D_2$, where the $D_i$ are monic diagrams $:t\to rd$. Recall that a
diagram from $t$ to $N$ ($t\leq N$) is monic if it has $t$ through arcs. One thinks of 
$D_1$ as the top half of $D$, and $D_2$ as the $^*$ of the bottom half of $D$,
where $^*$ is the cellular involution on the Temperley Lieb category, which
reflects diagrams in a horizontal line.
It follows that if we write $|\CB(d,r)|=b(d,r)$ and $|\CB(d,r;t)|=b(d,r;t)$, where 
$\CB(d,r;t)$ is the set of monic planar diagrams $D:t\to rd$ such that $L(D)\subseteq
\{d,2d,\dots,(r-1)d\}$, then 

\be\label{eq:dss}
b(d,r)=\sum_{0\leq t\leq dr}b(d,r;t)^2.
\ee

Now consider the right side of \eqref{eq:dims}. Define the positive integers 
$m(d,r;t)$ by
\be\label{eq:defm}
\Delta_K(d)^{\ot r}\cong \oplus_{t=0}^{dr}m(d,r;t)\Delta_K(t).
\ee

Thus the $m(d,r;t)$ are multiplicities, and $m(d,r;t)=0$ unless $t\equiv rd(\text{(mod $2$)}$.
Moreover, evidently, we have, if $m(d,r):=\dim_K\left(\End_{\U_K}(\Delta_K(d)^{\ot r}\right)$,
\be\label{eq:mss}
m(d,r)=\sum_{0\leq t\leq dr}m(d,r;t)^2.
\ee

It is clear that in view of \eqref{eq:dss} and \eqref{eq:mss}, the Theorem will follow if we
prove that for all $d,r$ and $t$,

\be\label{eq:fin}
m(d,r;t)=b(d,r;t).
\ee

We shall prove \eqref{eq:fin} by induction on $r$. If $r=1$, then
\be\label{eq:case1}
m(d,1;t)=b(d,1;t)=
\begin{cases}
0\text{ if }t\neq d\\
1\text{ if }t= d.\\
\end{cases}
\ee

Now by the Clebsch-Gordan formula, we have, for any integer $n$, 
$$\Delta_K(d)\ot\Delta_K(n)\cong \Delta_K(d+n)\oplus\Delta_K(d+n-2)\oplus\dots
\oplus\Delta_K(|d-n|).$$

It follows that
\be\label{eq:mrec}
m(d,r+1;t)=\sum_{s=t-d}^{t+d}m(d,r;s),
\ee
where $m(d,r;s)=0$ if $s<0$ or if $s>dr$.

We shall complete the proof of the Lemma by showing that the numbers $b(d,r;t)$ satisfy
a recurrence analogous to \eqref{eq:mrec}. For this observe that any diagram $D\in\CB(d,r;k)$
gives rise to a unique diagram in $\CB(d,r+1;k+d-2i)$,  for $0\leq i\leq \min\{d,k\}$, 
as depicted in Figure 3, and each diagram  $D' \in \CB(d,r+1;t)$ arises in this way from a unique
diagram in $\CB(d,r;k)$ for a uniquely determined $k$.  In fact, $k = t-d+2i$ where §i§ is the number of arcs in $D'$ whose
right vertices belong to $\{dr+1, \cdots, d(r+1)\}$. It follows that 

\be\label{eq:brec}
b(d,r+1;t)=\sum_{s=t-d}^{t+d}b(d,r;s),
\ee
where $b(d,r;s)=0$ if $s<0$ or if $s>dr$.

\begin{center}
\begin{tikzpicture}[scale=1.2]


\foreach \x in {1,7,8,10,11,13}
\filldraw(\x,2) circle (0.05cm);

\draw (1,2)--(7,2);
\draw (1,1.5)--(7,1.5);
\draw (1,2)--(1,1.5);
\draw (7,2)--(7,1.5);
\draw (1.5,1.5)--(1.5,0);
\draw (3.5,1.5)--(3.5,0);
\draw (11,2)--(11,0);
\draw (13,2)--(13,0);

\draw node[above] at (1,2){\tiny 1};
\draw node[above] at (7,2){\tiny{dr}};
\draw node[above] at (8,2){\tiny dr+1};\draw node[above] at (10,2){\tiny{dr+i}};\draw node[above] at (11,2){\tiny{dr+i+1}};\draw node[above] at (13,2){\tiny{d(r+1)}};


\draw(8,2) .. controls (8,0) and (6.5,0) .. (6.5,1.5);
\draw(10,2) .. controls (10,-1)  and (4.5,-1) .. (4.5,1.5);


\draw(2.5,0.7) node {\huge{$\cdots$}};
\draw(2.5,1) node {\small{k-i}};
\draw(12,0.7) node {\huge{$\cdots$}};
\draw(12,1) node {\small{d-i}};
\draw(5.8,0.7) node {\huge{$\cdots$}};
\draw(5.8,1) node {\small{i}};
\draw(8.6,0.9) node {\huge{$\cdots$}};
\draw(8.6,1.2) node {\small{i}};
\draw(4,1.75) node {\small{D}};
\end{tikzpicture}
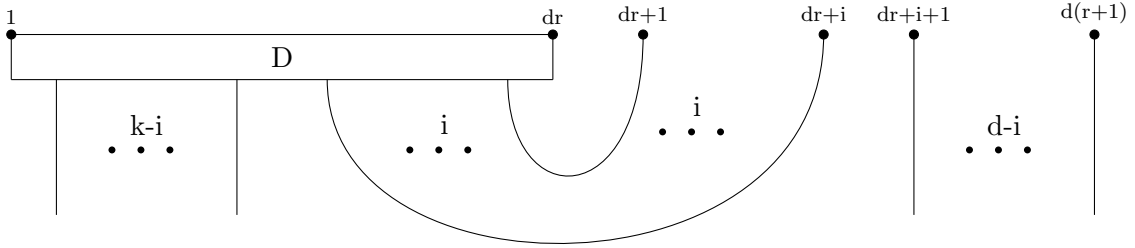
\centerline{Figure 3}
\end{center}
\label{fig4}

Comparing \eqref{eq:mrec} with \eqref{eq:brec}, and taking into account \eqref{eq:case1}, it follows that
$m(d,r;k)=b(d,r;k)$ for all $d,r$ and $k$. This completes the proof of \eqref{eq:fin} above, and hence of the Lemma.
\end{proof}

\subsection{Cellular structure}\label{ss:cell} We shall now complete the 
\begin{proof}[Proof of Theorem \ref{thm:cell}.] We have seen that $E_r(d,\tA)\cong p\TL_{rd}(\tA)p$,
and that the latter algebra has the basis $\CB(d,r)$, as stated in the theorem. It 
remains only to show that $p\TL_{rd}(\tA)p$ has a cellular structure.
Following \cite[Def. (1.1)]{GL96} we need to produce a cell datum $(\Lambda,M,C,^*)$ for
$p\TL_{rd}(\tA)p$. 

Take $\Lambda$ to be the poset $\{t\in\Z\mid 0\leq t\leq dr\text{ and }dr-t\in 2\Z\}$,
ordered as integers. For $t\in\Lambda$, let $M(t):=\CB(d,r;t)$, the set of monic planar
diagrams $D:t\to dr$ such that $L(D)\subseteq\{d,2d,\dots,(r-1)d\}$ 
(see \S \ref{ss:tl} and the proof of Lemma \ref{lem:basis}). Then the map 
$C:\amalg_{t\in\Lambda}M(t)\times M(t)\lr p\TL_{rd}(\tA)p$ is defined by
$C(D_1,D_2)=pD_1\circ D_2^*p$. Since each diagram $D\in\CB(r,d)$ is expressible
uniquely as $D=D_1\circ D_2^*$ for some $t\in\Lambda$ and $D_1,D_2\in M(t)$, 
it follows from Lemma \ref{lem:basis} that $C$ is a bijection from 
$\amalg_{t\in\Lambda}M(t)\times M(t)$ to a basis of $p\TL_{rd}(\tA)p$.
Finally, the anti-involution $^*$ is the restriction to $p\TL_{rd}(\tA)p$ of
the anti-involution on $\TL_{dr}(\tA)$, viz. reflection in a horizontal line.
Since $p^*=p$, we have $C(D_1,D_2)^*=(pD_1D_2^*p)^*=pD_2D_1^*p=C(D_2,D_1)$.

If $S,T\in M(t)$, we shall write $C(S,T)=C^t_{S,T}$, and for this proof only,
write $\CA=p\TL_{rd}(\tA)p$, and $\CA(<i)=\sum_{j<i,\; S,T\in M(j)}\tA C^j_{S,T}$.

It now remains only to prove the axiom (C3) of \cite[Def. (1.1)]{GL96}.
For this, let $S_1,S_2\in M(s)$ and $T_1,T_2\in M(t)$. Then 
\be\label{eq:cellax}
C^s_{S_1,S_2}C^t_{T_1,T_2}=pS_1(S_2^*pT_1)T_2^*p,
\ee
so that if $s<t$, the left side is in $\CA(<t)$, and there is nothing to prove.
Hence we take $s\geq t$.

Now $S_2^*pT_1$ is a morphism from $t$ to $s$, and hence is an $\tA$-linear combination 
of planar diagrams $D$ from $t$ to $s$. Thus the left side of \eqref{eq:cellax}
is an $\tA$-linear combination of elements of the form $pS_1DT_2^*p$. If $D$ is not
monic, then $pS_1DT_2^*p\in\CA(<t)$; if $D$ is monic, then clearly $pS_1DT_2^*p=
pS'T_2^*p$ for some monic $S':t\to dr$.

It follows from \eqref{eq:cellax} that modulo $\CA(<t)$, 
$C^s_{S_1,S_2}C^t_{T_1,T_2}=\sum_{S\in\CB(d,r;t)}a(S)C^t_{S,T_2}$, and 
$a(S)$ is independent of $T_2$. This proves the axiom (C3), and hence the cellularity of
$\CA$. The proof of Theorem \ref{thm:cell} is now complete.
\end{proof}

\section{Endomorphism algebras and specialisation.}
We shall prove in this section results showing how the multiplicities of the indecomposable summands of the specialisations of $\Delta_A(d)^{\otimes r}$ corresponding to homomorphisms $A\to k$ where $k$ is a field, relate to the dimensions of the simple modules for the corresponding endomorphism rings. It turns out that this is a consequence of a result on tilting modules which is valid for general quantum groups. Therefore in the first two subsections we deal with this general situation. Then in the last subsection we deduce the explicit consequences in our $\fsl_2$-case where we take advantage of our cellularity result from Section 4 on the endomorphism rings.

\subsection{Integral endomorphism algebras and specialisation} 
We now provide some rather general base change results for Hom-spaces between certain representations of quantum groups. So in this subsection we shall work with a general quantum group $U_q$ over $K$ with integral form $U_A$. We denote by $k$ an arbitrary field (in this subsection $k$ may even be any commutative noetherian $A$-algebra) made into an $A$-algebra by specializing $q$ to $\zeta \in k\setminus\{0\}$ and set $U_\zeta = U_A \otimes_A k$.  When $M$ is a $U_A$-module we write $M_q$, respectively $M_\zeta$ for the corresponding $U_q$ and $U_\zeta$-modules.

For each dominant weight $\lambda$ we write
$\Delta_q(\lambda), \Delta_A(\lambda)$ and $\Delta_\zeta(\lambda)$ for the Weyl modules
for $\U_q$, $\U_A$ and $\U_\zeta$ respectively. Similarly, we have the dual
Weyl modules $\nabla_q(\lambda), \nabla_A(\lambda)$ and $\nabla_\zeta(\lambda)$
respectively. Then it is well known that, writing $w_0$ for the longest 
element of the Weyl group,
$$
\nabla_\zeta(\lambda)=\Delta_\zeta(-w_0\lambda)^*,
$$
and similarly for $\nabla_A(\lambda)$ and $\nabla_q(\lambda)$.

We shall make repeated use of the following result. For any two weights $\lambda,\mu\in X$, we have
\be\label{kempf}
\Ext_{\U_A}^i(\Delta_A(\lambda),\nabla_A(\mu))=
\begin{cases}
A\text{ if }\lambda=\mu\text{ and }i=0\\
0\text{ otherwise}.\\
\end{cases}
\ee
This is proved exactly as in the corresponding classical case, see e.g. \cite{Ja}, II.B.4 by invoking the
quantised Kempf vanishing theorem proved in general in \cite{SRH}.
\begin{lem}\label{lemma-kempf}
Let $M,N$ be $\U_A$-modules which are finitely generated as $A$-modules.
If $M$ has a filtration by $\Delta_A(\lambda)$'s and $N$ has a filtration by
$\nabla_A(\mu)$'s, then $\Hom_{\U_A}(M,N)$ is a free $A$-module of rank
equal to $\dim_{\bQ(q)}\Hom_{\U_q}(M_q,N_q)$. Further, we have 
$$
\Hom_{\U_\zeta}(M_\zeta,N_\zeta)\simeq\Hom_{\U_A}(M_A,N_A)\otimes_A k.
$$
\end{lem}
\begin{proof}
We have a spectral sequence with $E_2$-terms 
$$ E_2^{-p, q} = \Tor_p^{A}(\Ext_{U_A}^q(M,N), k)$$
converging to $\Ext_{U_\zeta}^{q-p}(M_\zeta, N_\zeta)$. By (\ref{kempf}) we have $E_2^{-p, q} = 0$ if either $q>0$ or $q = 0<p$. Hence the spectral sequence collapses and we can read off the result.
\end{proof}

\begin{cor}\label{sufft}
Let $V$ be a $\U_A$-module, which satisfies
\be\label{assump-filt}
V^*\otimes_A V\text{ has a $\nabla_A$-filtration.}
\ee
Then $\End_{\U_\zeta}(V_\zeta^{\otimes r})\simeq \End_{\U_A}(V^{\otimes r})\otimes_A k$.
\end{cor}
\begin{proof}
We have $\End_{\U_A}(V^{\otimes r})\simeq\Hom_{\U_A}(\Delta_A(0), 
(V^*\otimes V)^{\otimes r})$, because $\Delta_A(0)$ is the trivial 
$\U_A$-module $A$. By the assumption (\ref{assump-filt}), we may apply 
Lemma \ref{lemma-kempf} to obtain the statement.
\end{proof}

As usual we denote by $\rho$ half the sum of the positive roots. Recall the concept of strongly multiplicity free modules from \cite{LZ1}. 

To see that there are significant cases where the above result applies, we have

\begin{prop}\label{pos-wts}
Suppose $V=\Delta_A(\lambda)$ for some dominant weight $\lambda$.
Assume that $V_q$ is strongly multiplicity free, and that
$-w_0\lambda+\mu+\rho$ is dominant for each weight $\mu$ of $V$.
Then $V^*\otimes V$ has a $\nabla_A$-filtration.
\end{prop}
\begin{proof} Recall that $\U_A$ has a triangular decomposition
$\U_A=\U_A^+\U_A^0\U_A^-$, and each weight $\mu$ defines a
$1$-dimensional representation of the subalgebra $\U_A^0\U_A^-$,
which we also denote by $\mu$.

We have $V^*=\nabla_A(\lambda')$ where $\lambda' = -w_0 \lambda$. Moreover $\nabla_A$ is realised as the induction
functor $\Ind_{\U_A^0\U_A^-}^{\U_A}$. Hence by a standard property of induction,
$$
V^*\otimes V=\Ind_{\U_A^0\U_A^-}^{\U_A}(\lambda')\otimes V
=\Ind_{\U_A^0\U_A^-}^{\U_A}(\lambda'\otimes V),
$$
where in this formula the last occurrence of $V$ is its 
restriction to $\U_A^0\U_A^-$. Now the hypothesis that 
$V_q$ is strongly multiplicity free implies that the weights of
$V$ are linearly ordered. But the weights of $\lambda'\otimes V$
are $\{\lambda'+\mu\}$ where $\mu$ runs over the weights of $V$.
This set is therefore a linearly ordered chain, and accordingly,
$\lambda'\otimes V$ has a $\U_A^0\U_A^-$-module filtration
$$
0=F_0\subset F_1\subset \dots\subset F_{d}=\lambda'\otimes V,
$$
where $d=\dim V_q$, with the quotients $F_i/F_{i-1}$ running over
the $\U_A^0\U_A^-$-modules $\lambda'+\mu$. Our hypothesis, together
with (the quantised) Kempf's vanishing theorem imply that the higher (degree $>0$)
cohomology of the corresponding line bundles vanishes, and hence that
induction is exact on this filtration. We therefore have 
a corresponding filtration of $\U_A$-modules
$$
0\subset \nabla_A(F_1)\subset \dots\subset \nabla_A(F_{d})=
\nabla_A(\lambda'\otimes V)=V^*\otimes V.
$$
Since this is a $\nabla_A$-filtration, we may apply \ref{sufft}
to complete the proof.
\end{proof}

\begin{cor}\label{sl2ex}
The conclusion of Proposition \ref{pos-wts} holds in the following cases.
\begin{enumerate}
\item If $V$ is a Weyl module with minuscule highest weight. This includes the 
natural modules in types $A,C$ and $D$ (but not type $B$).
\item If $V$ is any Weyl module for $\U_A(\fsl_2)$.
\item If $V$ is the Weyl module in type $G_2$
with highest weight $2\alpha_1+\alpha_2$ where $\alpha_1$ and $\alpha_2$ denote the two simple roots with $\alpha_2$ long.
\end{enumerate}
\end{cor}
\begin{proof}
When $V$ is minuscule, it is well known that for any weight $\mu$
of $V$ and any root $\alpha$, we have $(\mu,{\alpha\check~})=\pm 1$ or $0$, and hence (1) 
is clear. The case of $\fsl_2$ is evident, while in the case of type $G_2$,
the weights of the Weyl module in question are the short roots, together with $0$. This easily gives (3).
\end{proof}

\subsection{Multiplicities of tilting modules and dimensions of irreducibles.}

In this subsection we shall prove some rather general results which will allow us to relate multiplicities of indecomposable tilting summands in tensor powers of certain representations of quantum groups to the dimensions of simple modules for the corresponding endomorphism algebras.

\begin{thm}\label{fitting}
Let $k$ be a field, $\U$ a $k$-algebra, and $M$ a finite dimensional (over $k$)
$\U$-module. Let $E=\End_\U(M)$, and assume that for each indecomposable 
direct summand $M'$ of $M$, we have $E'/\Rad E' \simeq k$ where 
$E' = \End_{\U}(M')$. Then
$$
\frac{E}{\Rad E}\simeq \oplus_i M_{d_i}(k),
$$
where $M_d(k)$ is the algebra of $n\times n$ matrices over $k$, $i$ runs over
the isomorphism classes of indecomposable $\U$-modules (of course only a finite number occur),
and the $d_i$ are the multiplicities of the indecomposable summands of $M$.
\end{thm}
\begin{proof}
Let $M=M_1\oplus M_2\oplus\dots\oplus M_n$ be a decomposition of $M$ into indecomposables.
Then any endomorphism $\phi\in E$ may be written $\phi=(\phi_{ij})_{1\leq i,j\leq n}$,
where $\phi_{ij}\in\Hom_\U(M_j,M_i)$.

Now by Fitting's Lemma, any endomorphism of $M_i$ is either an automorphism or
is nilpotent. It follows that for each $i$, the set $R_i:=\{\psi\in E_i(:=\End_\U(M_i))\mid
\psi\text{ is not an automorphism}\}$ is a nilpotent ideal of $E_i$. In particular
there is an integer $N_i$ such that $R_i^{N_i}=0$.

Next, suppose that we have a sequence $i=i_1,i_2,\dots,i_{p+1}=i$, and 
$\phi_j:=\phi_{i_j,i_{j+1}}\in\Hom_\U(M_{i_{j+1}},M_{i_j})$ for $j=1,2,\dots,p$.
Consider $\psi_1:=\phi_1\dots\phi_{p-1}\phi_p\in\Hom_\U(M_i,M_i)$.
We shall show that 

\be\label{fittingcomp}
\begin{aligned}
&\text{\it If $\psi_1$ is an automorphism, then the $M_{i_j}$ are all isomorphic,}\\
&\text {\it  and $\phi_j$ is an isomorphism for each $j$.}\\
\end{aligned}
\ee

To see (\ref{fittingcomp}), let $\psi_j=\phi_j\dots\phi_p\phi_1\dots\phi_{j-1}
\in\Hom(M_{i_j},M_{i_j})$. If $\psi_j$ is an automorphism for each $j$, then 
for each $j$, $\phi_{j-1}$ is injective and $\phi_j$ is surjective, whence each 
$\phi_j$ is an automorphism, and we are done. If not, then there is some $j$ such
that $\psi_j$ is nilpotent. It follows that $\psi_1^N=0$ for large $N$, a contradiction.
This proves (\ref{fittingcomp}).

Now let $J$ be the subspace of $E$ consisting of the endomorphisms $\phi$ such that
$\phi_{ij}$ is not invertible for each pair $i,j$. If $E_i:=\End_\U(M_i)$
and $J_{ij}:=\{\phi_{ij}\in\Hom_\U(M_j,M_i)\mid \phi_{ij} \text{ is not invertible}\}$,
then again by Fitting's Lemma, $J_{ij}$ is an $(E_i,E_j)$ bimodule, and using
the observation (\ref{fittingcomp}) above, it is clear that $J$ is an ideal of $E$.
We shall show that $J$ is nilpotent.

Let $\phi^{(1)},\dots,\phi^{(\ell)}$ be a sequence of elements of $J$. Then 
$$
(\phi^{(1)}\dots\phi^{(\ell)})_{ij}=
\sum_{k_1,k_2,\dots,k_{\ell-1}}\phi_{ik_1}^{(1)}\phi_{k_1k_2}^{(2)}\dots\phi_{k_{\ell-1}j}^{(\ell)},
$$
where the sum is over all sequences $k_1,k_2,\dots,k_{\ell-1}$, with $1\leq k_i\leq n$
for all $i$.

Now we have seen that for any $j$, if $R_j=\Rad E_j$, then there is an integer $N_j$
such that $R_j^{N_j}=0$. If we take $\ell\geq N_1+N_2+\dots+N_n+2$, then there some index $a$
which occurs among the $k_i$ at least $N_a+1$ times. Then each summand in the
expression for $(\phi^{(1)}\dots\phi^{(\ell)})_{ij}$ contains a product of $N_a$
non-invertible elements of $E_a$ for some $a$, and hence is $0$.
Thus $J^{N_1+\dots+N_n+2}=0$.

Finally, it is clear that since we have assumed that
$E_i/R_i\simeq k$ for each $i$, $E/J\simeq\oplus_{i=1}^nM_{d_i}(k)$.
\end{proof}

The proof above actually yields more.
\begin{cor}\label{fittinggen}
Let $M$ be as in Theorem \ref{fitting} but drop the assumption on the endomorphism rings of direct summands of $M$.
Then there are division rings $D_i$ over $k$
such that 
$$
\frac{E}{\Rad E}\simeq \oplus_i M_{d_i}(D_i).
$$
\end{cor}
\begin{proof}
In this case Fitting's Lemma yields that $E_i/R_i$ is a division algebra $D_i$
over $k$, and the argument above proves the assertion.
\end{proof}

The application to our situation arises through the following property of tilting modules for quantum groups. We let $k$ be a field considered as an $A$-algebra via $q \mapsto \zeta \in k\setminus\{0\}$ and let $U_\zeta$ be as in Section 5.1. Then

\begin{prop}\label{tilt-endo}
Let $M$ be an indecomposable tilting module for $\U_\zeta$ and set  
$E = \End_{\U_\zeta}(M)$. Then  $E/\Rad E \simeq k$. 
\end{prop}
\begin{proof}
By the Ringel-Donkin classification \cite{Do93} (see \cite{A-tilt} for the adaption to the quantum case) of indecomposable tilting modules we get
that $M$ has a unique highest weight $\lambda \in X^+$ and that the weight space $M_\lambda$ is $1$-dimensional. 
Therefore any $\varphi \in \End_{\U_\zeta}(M)$ is given by a scalar $a \in k$ on 
$M_\lambda$. But then $\varphi - a\ \id_M$ is not an
automorphism, i.e. $\varphi - a\  \id_M \in \Rad E$.
\end{proof}

We denote the indecomposable tilting module for $U_\zeta$ with highest weight $\lambda$ by $\CT_\zeta(\lambda)$ and for an arbitrary tilting module $\CT$ for $U_\zeta$ we write $(\CT:\CT_\zeta(\lambda))$ for the multiplicity with which $\CT_\zeta(\lambda)$ occurs as a summand of $\CT$. Then 
Theorem \ref{fitting} together with Proposition \ref{tilt-endo} give

\begin{cor} \label{tilt-simp}
For any tilting module $\CT$ for $U_\zeta$ and any $\lambda \in X^+$ we have
$$ 
(\CT:\CT_\zeta(\lambda)) = \dim_k L_\zeta(\lambda),$$
where $L_\zeta(\lambda)$ is the simple module for the algebra $E = \End_{U_\zeta}(\CT)$ corresponding to $\lambda$.
\end{cor}

\subsection{Multiplicities for $\U_\zeta(\fsl_2)$}
We now apply the above general results to $\fsl_2$. With $k$ and $\zeta$ as above, the indecomposable  tilting modules in this case are $\CT_\zeta(m)$ with $m \in \N$. If $\zeta$ is not a root of unity in $k$ then the category of finite dimensional $U_\zeta$-modules is semisimple and behaves exactly like the corresponding category for the generic quantum group $U_q$. 

From now on we assume that $\zeta$ is a root of unity; for the specialisation $\U_\zeta$ etc., we assume that
the homomorphism $A\to k$ is given by $q\mapsto \zeta$ (so $q^{\frac{1}{2}}\mapsto \sqrt\zeta$) and we set $\ell = \ord (\zeta^2)$. 
If $d$ is a positive integer with $d < \ell$ we have $\Delta_\zeta(d) = \CT_\zeta(d)$ and all the 
tensor powers $\CT_r = \Delta_\zeta(d)^{\otimes r}$ are also tilting modules. 
We set $E_\zeta(d,r) = \End_{U_\zeta}(\CT_r)$. By Lemma \ref{lemma-kempf} we have
$$ E_\zeta(d,r) = E_r(d, \tA)\otimes _{\tA} k,$$
where as before $\tA = A[([d]!)^{-1}]$. Note that our assumption $\ell>d$ ensures that the specialization $\phi_\zeta : A \rightarrow k$ factors through $\tA$ making $k$ into an $\tA$-algebra.

Our cellularity results from Section 3 imply that
\be\label{eq:endzeta}
E_\zeta(d,r) \cong p_\zeta\TL_{dr}(k)p_\zeta,
\ee
where $p_\zeta$ is the specialisation at $q=\zeta$ of the idempotent 
$p\in\TL_{dr}(\tA)$. Note that in $\TL_{dr}(k)=\TL_{dr,\zeta}(k)$ the generators $f_i$ satisfy $f_i^2=(\zeta+\zeta\inv)f_i$.

The simple modules for the cellular algebra $p_\zeta\TL_{dr}(k)p_\zeta$ are parametrised by the poset $\Lambda = \{m \in \Z \mid 0 \leq m \leq dr \text { and } dr - m \in 2\Z\}$, see Section 4.2. We denote the simple module associated with $m \in \Lambda$ by $L_\zeta(m)$. Then we get

\begin{thm}\label{thm:mult}
In the above notation we have for $m \in \Lambda$
$$ (\CT_r : \CT_\zeta(m)) = \dim_k L_\zeta(m).$$
This multiplicity is the rank of the matrix whose rows and columns are labelled by $\CB(d,r;m)$ 
(see Section 4.1) and whose $D_1,D_2$ entry is the coefficient of the identity map $m\to m$ 
(in the Temperley Lieb category) in the expansion  of $D_2^*p_\zeta D_1$ as a linear combination
of diagrams from $m$ to $m$.
\end{thm}

\begin{proof}
The equality in the theorem is an immediate consequence of Corollary \ref{tilt-simp}. To see the second statement note that 
$L_\zeta(m)$ is realised as follows. 
Let $W_\zeta(m)$ be the cell module corresponding to $m$. This has $k$-basis $C_S$, 
$S\in \CB(d,r;m)$,
the monic diagrams $D$ from $m$ to $dr$ such that $L(D)\subseteq\{d,2d,\dots,(r-1)d\}$.
We may think of $C_S$ as $p_\zeta S$, and then the
$E_\zeta(d,r)$-action is by left composition: for $x\in E_\zeta(d,r)$, 
$x.C_S=\sum_{T\in\CB(d,r;m)} a(T,S)C_{T}$, where 
$$
xp_\zeta S= \sum_{T\in\CB(d,r;m)} a(T,D)p_\zeta T+\text{ lower terms},
$$
where ``lower'' means ``having fewer through arcs''.
  
There is an invariant form $(-,-)$ on $W_\zeta(m)$, defined by
\be\label{eq:form}
{C^m_{S,T}}^2\in (C_S,C_T)C^m_{S,T}+E_\zeta(d,r)(<m)\text{ for $S,T$ in }\CB(d,r;m).
\ee
The radical $\Rad_\zeta(m)$ of this form is a submodule of $W_\zeta(m)$, and 
$$
L_\zeta(m)
= W_\zeta(m)/\Rad_\zeta(M).
$$
It is therefore evident that $\dim L_\zeta(m)$ is equal to
the rank of the matrix $M_{m,\zeta}$, whose rows and columns are indexed by $\CB(d,r;m)$,
and whose $(S,T)$-entry is $(C_S,C_T)$.

Finally, since ${C^m_{S,T}}^2=p_\zeta S(T^*p_\zeta S)T^*p_\zeta$, and noting that
$T^*p_\zeta S$ is a linear combination of diagrams $:m\to m$, it follows
from \eqref{eq:form} that $(C_S,C_T)$ is the coefficient of $\id:m\to m$.
\end{proof}

Since $\dim W_\zeta(dr)=1$ and the coefficient of $\id:d\to d$ in $p_d(\zeta)$
is $1$, it is immediate from the Theorem that the multiplicity of $\CT_\zeta(dr)$ 
is $1$. We finish this section with a less trivial example.

\begin{example}\label{ex:sub}
Take $k=dr-2$. We shall compute the multiplicity of $\CT_\zeta(k)$ in
$\Delta_\zeta(d)^{\ot r}$ for any $d,r$. Here $\CB(d,r;dr-2)=\{S_1,S_2,\dots,S_{r-1}\}$,
where $S_i$ is as shown in Figure 4.

\begin{center}
\begin{tikzpicture}[scale=1.2]


\foreach \x in {1,2,5,6,7,8,11,12}
\filldraw(\x,2) circle (0.05cm);

\draw (1,2)--(2,2);\draw(5,2)--(6,2);\draw(7,2)--(8,2);\draw(11,2)--(12,2);
\draw (1,1.5)--(2,1.5);\draw(5,1.5)--(6,1.5);\draw(7,1.5)--(8,1.5);\draw(11,1.5)--(12,1.5);
\draw (1,2)--(1,1.5);\draw (2,2)--(2,1.5);\draw (5,2)--(5,1.5);\draw (6,2)--(6,1.5);
\draw (7,2)--(7,1.5);\draw (8,2)--(8,1.5);\draw (11,2)--(11,1.5);\draw (12,2)--(12,1.5);

\draw (1,1.5)--(1,0);
\draw (2,1.5)--(2,0);
\draw (5,1.5)--(5,0);
\draw (8,1.5)--(8,0);
\draw (11,1.5)--(11,0);
\draw (12,1.5)--(12,0);

\draw node[above] at (1,2){\tiny 1};
\draw node[above] at (6,2){\tiny{id}};
\draw node[above] at (7,2){\tiny id+1};
\draw node[above] at (12,2){\tiny{dr}};
\draw node[above] at (2,2){\tiny{d}};\draw node[above] at (5,2){\tiny{(i-1)d+1}};
\draw node[above] at (8,2){\tiny{d(i+1)}};\draw node[above] at (11,2){\tiny{(r-1)d+1}};


\draw(6,1.5) .. controls (6.3,0.3) and (6.7,0.3) .. (7,1.5);


\draw(3.5,1.75) node {{$\cdots$}};
\draw(9.5,1.75) node {{$\cdots$}};
\draw(1.5,1.75) node {\small{$p_d(\zeta)$}};\draw(5.5,1.75) node {\small{$p_d(\zeta)$}};
\draw(7.5,1.75) node {\small{$p_d(\zeta)$}};\draw(11.5,1.75) node {\small{$p_d(\zeta)$}};
\end{tikzpicture}
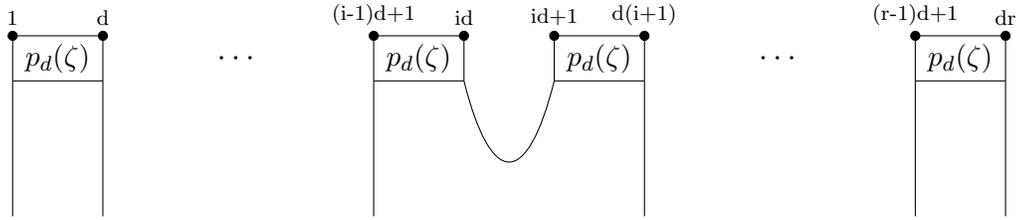
\centerline{Figure 4}
\end{center}
\label{fig5}

Now by repeated use of the diagrammatic recursion
\begin{center}
\begin{tikzpicture}[scale=1.2]



\draw (1,1.5)--(2,1.5);\draw(5,1.5)--(6,1.5);\draw(9,2.5)--(10,2.5);\draw(9,2)--(10,2);
\draw (1,1)--(2,1);\draw(5,1)--(6,1);\draw(9,0.5)--(10,0.5);\draw(9,0)--(10,0);
\draw (1,1)--(1,1.5);\draw (2,1)--(2,1.5);\draw (5,1)--(5,1.5);\draw (6,1)--(6,1.5);\draw (10,2)--(10,2.5);
\draw (9,2)--(9,2.5);\draw (4.8,1)--(4.8,1.5);
\draw (9,0)--(9,0.5);\draw (10,0)--(10,0.5);
\draw (10,0.5)--(10,2);
\draw (9.35,0.5)--(9.35,2);



\draw(9,0.5) .. controls (8.5,1) and (8.2,0.3) .. (8,0);
\draw(9,2) .. controls (8.5,1.5) and (8.2,2.2) .. (8,2.5);


\draw(7,1.25) node {{$-\frac{[d-1]}{[d]}$}};
\draw(-1,1.25) node {{(*)}};
\draw(3,1.25) node {{$=$}};
\draw(1.5,1.25) node {\small{$p_d$}};\draw(5.5,1.25) node {\small{$p_{d-1}$}};
\draw(9.5,2.25) node {\small{$p_{d-1}$}};\draw(9.5,0.25) node {\small{$p_{d-1}$}};
\end{tikzpicture}
\end{center}
it is straightforward to compute the Gram matrix $M_{dr-2,\zeta}$ of the invariant form
(see the proof above). One shows that 
$$
(S_i,S_j)=
\begin{cases}
0\text{ if }j\neq i\text{ or }i\pm 1\\
\frac{[2]_{\zeta^d}}{[d]_\zeta}\text{ if }j=i\\
(-1)^{d+1}[d]_\zeta\inv\text{ if }j=i\pm 1.\\
\end{cases}
$$

Hence the Gram matrix of the invariant form is the matrix of size $(r-1)\times(r-1)$ shown below.

$$ 
M_{dr-2,\zeta}=
\frac{1}{[d]_\zeta}
\begin{pmatrix}
\delta &( -1)^{d+1} & 0&\hdots &\dots &\dots\\
(-1)^{d+1} &\delta&(-1)^{d+1} &0 &\hdots&\dots\\
0&(-1)^{d+1} &\delta &(-1)^{d+1} &0&\hdots\\
\vdots&&\ddots&\ddots&\ddots&\\
&&&&\ddots&(-1)^{d+1}\\
0&\dots&\dots&0&(-1)^{d+1}&\delta\\
\end{pmatrix},
$$
where $\delta=\zeta^d+\zeta^{-d}=[2]_{\zeta^d}$.

Now it is easily shown by induction that any $n\times n$ matrix of the form
$$ 
A=
\begin{pmatrix}
a_1 &b_1 & 0&\hdots &\dots &\dots\\
1&a_2&b_2 &0 &\hdots&\dots\\
0&1&a_3&b_3 &0&\hdots\\
\vdots&&\ddots&\ddots&\ddots&\\
&&&&\ddots&b_{n-1}\\
0&\dots&\dots&0&1&a_n\\
\end{pmatrix}
$$
with entries in a principal ideal domain, may be transformed by row and column operations into
$$ 
A'=
\begin{pmatrix}
1 &0 & 0&\hdots &\dots &\dots\\
0&1&0 &0 &\hdots&\dots\\
0&0&1&0 &0&\hdots\\
\vdots&&\ddots&\ddots&\ddots&\\
&&&\ddots&1&0\\
0&\dots&\dots&0&0&D\\
\end{pmatrix},
$$
where $D=\det(A)$.

It follows that the rank of the Gram matrix $M_{dr-2,\zeta}$ is $r-1$ if $\det M_{dr-2,\zeta}\neq 0$,
while if $\det M_{dr-2,\zeta}=0$, the rank is $r-2$.

Now the determinant of $[d]_\zeta M_{dr-2,\zeta}$ is easily computed
(cf. \cite[(6.18.2)]{GL96}), and using this, we see that
$$
\det M_{dr-2,\zeta}=(-1)^{(d+1)(r+1)}([d]_\zeta)^{-(r-1)}[r]_{(-1)^{d+1}\zeta^d}.
$$
It therefore follows that the multiplicity of $\CT_\zeta(dr-2)$
in $\Delta_\zeta(d)^{\otimes r}$ is 

$$
\begin{cases}
r-1\text{ if }[r]_{(-1)^{d+1}\zeta^d}\neq 0\\
r-2\text{ otherwise.}
\end{cases}
$$
Finally, observe that $[r]_{(-1)^{d+1}\zeta^d}=0\iff \zeta^{2dr}=1$. Hence if we write 
(using the convention that for any root of unity $\xi$, we denote by $|\xi|$ or by $\ord(\xi)$
the multipliciative order of $\xi$)
\be\label{eq:ell}
\ell=\begin{cases}
|\zeta|\text{ if $|\zeta|$ is odd}\\
\frac{|\zeta|}{2}\text{ if $|\zeta|$ is even}\\
\end{cases},
\ee
then $\ell=|\zeta^2|$, whence the multiplicity of  $\CT_\zeta(dr-2)$
in $\Delta_\zeta(d)^{\otimes r}$  is given by
\be\label{eq:ex}
(\CT_r:\CT_\zeta(dr-2)) =
\begin{cases}
r-1\text{ if }\ell\not | dr\\
r-2\text{ if }\ell|dr.\\
\end{cases}
\ee
This shows also by standard cellular theory that the cell module $W_\zeta(dr-2)$ of $E_\zeta(d,r)$) is simple if 
$\ell\not | dr$, while if $\ell | dr$, then $W_\zeta(dr-2)$ has composition factors $L_\zeta(d,r;dr-2)$ and
$L_\zeta(d,r;dr)$, (the latter being the trivial module), each with multiplicity one.
\end{example}

\section{Complex roots of unity.}

In this section we take $k = \C$ and fix a root of unity $\zeta\in\C$. As before we set  $\ell= \ord(\zeta^2)$. In this case the
structure of the tilting modules $\CT_\zeta(m)$ is well understood, and hence provides
an alternative approach to the computation
of the multiplicities $\mu_\zeta(d,r;m) := (\Delta_\zeta(d)^{\otimes r} : \CT_\zeta(m))$, and hence of the dimensions of the
simple modules for the cellular algebra $E_\zeta(d,r)$  (see Theorem \ref{thm:mult}). In this section we
demonstrate how this is done. We then show how these results on tilting modules may alternatively be deduced from 
results on the decomposition numbers of the algebras $E_\zeta(d,r)$, which are also proved in this section.

\subsection{Structure of tilting modules}

\begin{prop}\label{prop:tilting}
The indecomposable tilting module $\CT_\zeta(m)$ for $\U_\zeta=\U_\zeta(\fsl_2)$ with
highest weight $m$ has the following description.
\begin{enumerate}
\item If either $m<\ell$ or $m\equiv -1(\mod \ell)$ then $\CT_\zeta(m)\simeq\Delta_\zeta(m)$ is irreducible.
\item Write $m=a\ell+b$, where $a\geq 1$ and $0\leq b<\ell-1$. Then $\CT_\zeta(m)$ is the unique non-trivial
extension
$$
0\lr\Delta_\zeta(m)\lr\CT_\zeta(m)\lr \Delta_\zeta(m-2b-2)\lr 0.
$$
\end{enumerate} 
\end{prop}

\begin{proof} This result is certainly well known. As we haven't been able to find a reference where this is explicitly stated we sketch the easy proof.

Denote by $\CL_\zeta(m)$ the simple $U_\zeta$-module with highest weight $m \in \N$ (not to be confused with the simple $E_\zeta(d,r)$-module $L_\zeta(m)$). It follows from the strong linkage principle \cite{A-SL} (or by direct calculations) that $\CL_\zeta(m) = \Delta_\zeta(m)$ if and only if $m$ satisfies the conditions in (1); in particular, (1) holds. 

So assume $m = a\ell + b$ with $a, b$ as in (2). The tensor product $\Delta(a\ell -1) \otimes_\C \Delta_\zeta(b+1)$ has a Weyl filtration with factors $\Delta_\zeta(m), \Delta_\zeta(m-2), \cdots , \Delta_\zeta(m-2(b+1))$. Note that the first and the last factors belong to the same linkage class and that none of the other factors are in this class. Hence by the linkage principle (loc. cit.) there is a summand  $\CT$ of $\Delta_\zeta(a\ell -1) \otimes_\C \Delta_\zeta(b+1)$ which has these two Weyl factors, i.e. fits into an exact sequence
$$ 0\lr\Delta_\zeta(m)\lr\CT\lr \Delta_\zeta(m-2b-2)\lr 0.$$
By case (1) we see that $\Delta_\zeta(a\ell -1) \otimes_\C \Delta_\zeta(b+1)$ is tilting. Hence so is our summand $\CT$. 
The proof of case (2) will therefore be complete if we prove that $\CT$ is indecomposable. This in turn would follow if there were no non-trivial homomorphisms $\CT$ of $\Delta_\zeta(a\ell -1) \otimes_\C  \Delta_\zeta(b+1) \lr \CL_\zeta(m)$.  To check the last statement, we need the quantised Steinberg tensor product theorem,
\cite{AW} Theorem 1.10, for simple modules, $\CL_\zeta(m) \simeq \CL_\zeta(a\ell)\otimes \CL_\zeta(b)$ (again in the case at hand this can alternatively be checked by direct calculations). 

Using this together with the self-duality of the simple modules, and the result in (1) we get
$\Hom_{U_\zeta} (\Delta_\zeta(a\ell-1) \otimes_\C \Delta_\zeta(b+1), \CL_\zeta (m)) \simeq 
\Hom_{U_\zeta} (\CL_\zeta(a\ell-1) \otimes_\C \CL_\zeta(b+1), \CL_\zeta (m)) \simeq
\Hom_{U_\zeta} (\CL_\zeta((a-1)\ell) \otimes_\C \CL_\zeta(\ell-1)\otimes_\C \CL_\zeta(b+1), \CL_\zeta (a\ell) \otimes_\C\CL_\zeta(b)) \simeq
\Hom_{U_\zeta} (\CL_\zeta((a-1)\ell) \otimes_\C \CL_\zeta(b+1)\otimes_\C \CL_\zeta(b), \CL_\zeta (a\ell) \otimes_\C\CL_\zeta(\ell -1)) \simeq
\Hom_{U_\zeta} (\CL_\zeta((a-1)\ell) \otimes_\C \CL_\zeta(b+1)\otimes_\C \CL_\zeta(b), \CL_\zeta ((a+1)\ell -1)).
$
Note that the last Hom-space is $0$ because by our condition on $b$ the weight $(a+1)\ell -1$ is strictly larger than all weights of $\CL_\zeta((a-1)\ell) \otimes_\C \CL_\zeta(b+1)\otimes_\C \CL_\zeta(b)$.
\end{proof}
Since the the weights of $\Delta_\zeta(m)$ are $m, m-2, \cdots , -m$, each occuring with multiplicity one  we deduce
\begin{cor}\label{cor:dimwt}
We have
$$
\dim \CT_\zeta(m)_t=
\begin{cases}
1\text{ if }t=m-2i, \; 0\leq i\leq m \text { in case (1)}\\
2\text{ if }t=m-2j, \; b+1\leq j\leq m-(b+1) \text { in case (2)}\\
1\text{ if }t=m-2j,\text{ with } 0\leq j\leq b\text{ or }m\geq j\geq m-b \text { in case (2)}\\
0\text{ otherwise.}
\end{cases}
$$
\end{cor}

\subsection{Mulplicities and dimensions}

Now the equation 
\be\label{eq:decomp}
\Delta_\zeta(d)^{\ot r}\cong\oplus_{m=0}^{dr}\mu_\zeta(d,r;m)\CT_\zeta(m).
\ee
may be used to relate the multiplicities to the dimensions of the weight spaces. For this purpose, we make the following definitions.

\begin{definition}
\begin{enumerate}
\item Let $w(d,r;m):=\dim(\Delta_\zeta(d)^{\ot r})_m$. This is independent of $\zeta$.
\item Let $a_m=a_m(d,r):=|\{(i_1,\dots,i_r)\mid 0\leq i_j\leq d\;\forall j,\;\;\sum_ji_j=m\}|$.
\end{enumerate}
\end{definition}
Note that $a_m = a_{dr-m}$ for all $m$.

\begin{lem}\label{lem:rec}
\begin{enumerate}
\item For $0\leq m\leq dr,\;\;m\equiv dr(\mod 2)$,  $w(d,r;m)=a_{\frac{m+dr}{2}}$.
\item We have
$$
w(d,r;m)=\mu_\zeta(d,r;m)+\sum_{j=1}^{\frac{dr-m}{2}}\dim \CT_\zeta(m+2j)_m\mu_\zeta(d,r;m+2j).
$$
\end{enumerate}
\end{lem}

The first statement follows easily from the fact that $\Delta_\zeta(d)^{\ot r}$ has $q$-character $[d+1]^r$,
while the second arises from \eqref{eq:decomp} by taking the dimension of the $m$-weight spaces on both sides,
taking into account that $\CT_\zeta(t)$ has only weights $m$ of the form $m=t-2i$, $i\geq 0$.

Lemma \ref{lem:rec} (2) may be used to determine the multiplicities $\mu_\zeta(d,r;m)$ recursively.
We shall do this for the case considered in Example \ref{ex:sub}. 

\begin{example}\label{ex:sub2}
Let us compute $\mu_\zeta(d,r,dr-2)$. By Lemma \ref{lem:rec} (2), $w(d,r;dr-2)=\mu_\zeta(d,r;dr-2)+\dim\CT_\zeta(dr)_{dr-2}$.
Moreover it follows from Corollary \ref{cor:dimwt} that
$$
\dim\CT_\zeta(dr)_{dr-2}=
\begin{cases}
2\text{ if }b=0\\
1\text{ if }b\neq 0.\\
\end{cases}
$$
Noting that by Lemma \ref{lem:rec} (1) we have $w(d, r, dr-2) = a_{dr-1} =a_1 = r$ we get 
$$
\mu_\zeta(d,r;dr-2)=
\begin{cases}
r-1\text{ if }\ell\not | dr\\
r-2\text{ if }\ell | dr,\\
\end{cases}
$$
in accord with \eqref{eq:ex}.
\end{example}

\begin{example}
In Example 6.5 we considered multiplicities $\mu_\zeta(d,r;t)$ where $t$ was large, namely $t =
dr-2$. We now consider the case where $t$ is small.

Assume $t< \ell$. Then we may apply Formula 3.20 (1) in \cite{AP}. Using the notation from Section 4.1 this formula reads in our case
$$ \mu_\zeta(d,r;t) = \sum_{j\geq 0} m(d,r;t+2j\ell) - \sum_{i>0} m(d,r; 2i\ell-t-2).$$
Recall that the multiplicities $m(d,r;t)$ are given by the recursion relation \eqref{eq:mrec}, i.e. they may be calculated by induction on $r$. 
\end{example}

In fact this formula is valid in general: maintaining the notation of Example 6.6 (except that the integer $t$ below may now be arbitrary) we have

\begin{prop}\label{prop:mu}
Let $t \in \N$. Then
\begin{enumerate}
\item[(1)] If $t \equiv -1 \quad (\mod \ell)$ then $\mu_\zeta(d,r;t) = m(d,r;t)$.

\item[(2)] If $t \not \equiv -1 \quad (\mod \ell)$ then, writing $t=a\ell+b$ with $0\leq b\leq\ell-2$, we have
$$ 
\begin{aligned}
\mu_\zeta(d,r;t) =& \sum_{j\geq 0} m(d,r;t+2j\ell) - \sum_{i\geq 1} m(d,r;t-2b-2+2i\ell)\\
=& \sum_{j\geq 0} m(d,r;t+2j\ell) - \sum_{i\geq a+1} m(d,r;2i\ell-t-2).\\
\end{aligned}
$$
\end{enumerate}
\end{prop}

\begin{proof} This follows easily from the description in Proposition \ref{prop:tilting} of the indecomposable tilting modules $\CT_\zeta(m)$, by taking characters in the relation $\Delta_\zeta(d)^{\ot r}\cong \oplus_m\mu(d,r;m)\CT_\zeta(m)$. Let $\CC_1$ be the set of positive integers in case (1)
of Proposition \ref{prop:tilting}, and similarly $\CC_2$ those in case (2).

 If we denote by $c_t$ the $q$-character of $\Delta_q(t)$,
then Proposition \ref{prop:tilting} shows that 
if $t\in\CC_1$, then $\ch(\CT_\zeta(t))=c_t$, while if $t\in\CC_2$, then $\ch(\CT_\zeta(t))=c_t+c_{t-2b-2}$.
Now substitute these values and compare coefficients of $c_t$ in the equation
$$
\sum_{t\in\N}m(d,r;t)c_t=\sum_{t\in\CC_1}\mu_\zeta(d,r;t)\ch(\CT_\zeta(t))+\sum_{t\in\CC_2}\mu_\zeta(d,r;t)\ch(\CT_\zeta(t)).
$$
One obtains $\mu_\zeta(d,r;t)=m(d,r;t)$ if $t\equiv -1(\mod\ell)$, while if $t=a\ell+b$ with $a\geq 0$ and $0\leq b\leq \ell-2$, we have
\be\label{eq:sum}
m(d,r;t)=\mu_\zeta(d,r;t)+\mu_\zeta(d,r;(a+2)\ell-b-2).
\ee
Now for any integer $t=a\ell+b\geq 0$ such that $t\not\equiv -1(\mod \ell)$, write $g(t)= (a+2)\ell-b-2$; then $g(t)\not\equiv-1(\mod\ell)$,
and the relation above reads $m(d,r;t)=\mu_\zeta(d,r;t)+\mu_\zeta(d,r;g(t))$. It follows that
$\mu_\zeta(d,r;t)=\sum_{i\geq 0}m(d,r;g^{2i}(t))-\sum_{j\geq 0}m(d,r;g^{2j+1}(t))$.
The statements (1) and (2) are now immediate.
\end{proof}

As these multiplicities are also dimensions of simple modules for our cellular algebra from Section 4 we 
may rewrite these formulae as follows (again using notation from Section 4.1)

\begin{cor} Let $t\in \N$. Then 
\begin{enumerate}
\item[(1)] If $t \equiv -1 \quad (\mod \ell)$ then $\dim_\C L_\zeta(t) = b(d,r;t)$.

\item[(2)] If $t \not \equiv -1 \quad (\mod \ell)$ then writing $t=a\ell+b$ with $0\leq b\leq\ell-2$, we have
$$
\dim_\C L_\zeta(t)  = \sum_{j\geq 0} b(d,r;t+2j\ell) - \sum_{i\geq a+1} b(d,r; 2i\ell-t-2).
$$
\end{enumerate}
\end{cor}

Note that the numbers $b(d,r;t)$ are dimensions of the cell modules of the cellular algebra $pTL_{dr}(\wt A)p$,
which do not change under specialisation.


\subsection{Decomposition numbers} In this subsection we shall determine the decomposition numbers of the cellular
algebra $E_\zeta(d,r)$, and show how the weight multiplicities of the tilting modules are determined by these,
giving an alternative proof of Corollary \ref{cor:dimwt}. The algebra has cell modules $W_\zeta(t)$ as described in \S 4.2
and $\dim(W_\zeta(t))=b(d,r;t)$. If $L_\zeta(t)$ is the corresponding simple module, we write $d_{st}=[W_\zeta(t):L_\zeta(s)]$
for the multiplicity of $L_\zeta(s)$ in $W_\zeta(t)$. It is known by the theory of cellular algebras that the matrix $(d_{st})$
is lower unitriangular.

We have $\dim(L_\zeta(t))=\mu_\zeta(d,r;t)$, and therefore we clearly have
\be\label{eq:dec}
b(d,r;t)=\sum_{s\geq t}d_{st}\mu_\zeta(d,r;s).
\ee

\begin{thm}\label{thm:dec} Maintain the notation above. Suppose $\ell\in\N$ is such that $\ell=\ord(\zeta^2)$ and write
$\N=\CN_1\amalg\CN_2$, where $\CN_1=\{t\in\N\mid t\equiv-1(\mod \ell)\}$ and 
$\CN_2=\N\setminus\CN_1$. Let $g:\CN_2\lr\CN_2$ be the function defined in the proof of
Proposition \ref{prop:mu}, viz. if $t=a\ell+ b$ with $0\leq b\leq\ell-2$, then $g(t)=(a+1)\ell+\ell-b-2$.
Observe that $g(t)=t+2(\ell-b-1)\geq t+2$, and that $g(t)\equiv t(\mod 2)$.
\begin{enumerate}
\item For each $t\in\CN_2$ such that $0\leq t<g(t)\leq dr$, $t\equiv dr(\mod 2)$, there is a non-zero homomorphism 
$\theta_t:W_\zeta(g(t))\lr W_\zeta(t)$, which is uniquely determined up to scalar multiplication.
\item The $\theta_t$ are the only non-trivial homomorphisms between the cell modules of $E_\zeta(d,r)$.
\item Let $t\in\N$ be such that $0\leq t\leq dr$ and $t\equiv dr(\mod 2)$. If $t\in\CN_2$ and $g(t)\leq dr$, then 
$W_\zeta(t)$ has composition factors $L_\zeta(t)$ and $L_\zeta(g(t))$, each with multiplicity $1$. All other
cell modules are simple.
\item The decomposition numbers of $E_\zeta(d,r)$ are all equal to $0$ or $1$.
\end{enumerate}
\end{thm}
Note that (3) and (4) are formal consequences of (1) and (2).
\begin{proof}

We begin by observing that the conjecture is true when $d=1$. In this case $E_\zeta(1,r)=\TL_{r,\zeta}(\C)$, the structure of whose cell modules
(as well as all homomorphisms between them) is treated in \cite{GL98}. In particular, \cite[Theorem 5.3]{GL98} asserts
that (in our notation above) if $s\neq t$, then $L_\zeta(s)$ is a composition factor of $W_\zeta(t)$ if and only if 
$s$ satisfies both (i) $t+2\ell>s>t$ and (ii) $s+t+2\equiv 0(\mod 2\ell)$. It is an easy exercise to show that (i) and (ii) are
equivalent to (iii) $t\not\equiv-1(\mod\ell)$ and (iv) $s=g(t)$. This yields all the statements of the theorem for this case.

Next recall that $E_\zeta(d,r)\cong p_d(\zeta)\TL_{dr,\zeta}(\C)p_d(\zeta)$, where $p_d(\zeta)$ is the specialisation at
$\zeta$ of the idempotent $p_d$. Thus we may define the exact functor $\CF_d:\Mod(\TL_{dr,\zeta}(\C))\lr\Mod(E_\zeta(d,r))$, 
by $M\mapsto p_d(\zeta)M$, where $\Mod$ indicates the category of left modules for the relevant algebra. Now it is evident
from the description in \S\ref{ss:cell} of the cell module $W(t)$ and its basis $\CB(d,r;t)$, that $\CF_d(W_{\TL_{dr,\zeta}(\C)}(t))
=W_{E_\zeta(d,r)}(t)$ for all $t$ with $0\leq t\leq dr$ and $t+dr\in 2\Z$. 

Moreover by exactness, for any simple 
$\TL_{dr,\zeta}(\C)$-module $L$, $\CF_d(L)$ is either a simple $E_\zeta(d,r)$-module or zero. It follows thus
(and also from the explicit diagrammatic description), that $\CF_d(L_{\TL_{dr,\zeta}(\C)}(t))=L_{E_\zeta(d,r)}(t)$ 
whenever the latter is non-zero. Given the description in \S\ref{ss:cell} of the cellular structure,
and the fact that $\TL_{dr,\zeta}(\C)$ 
is quasi-hereditary when $\zeta\neq \zeta_4=\exp(\frac{\pi i}{2})$, $\CF_d$ does
not kill any non-trivial simple $\TL_{dr,\zeta}(\C)$-module (this may be checked directly when $\zeta=\zeta_4$).
The quasi-heredity of $\TL_{dr,\zeta}(\C)$ when $\zeta\neq \zeta_4$ is well known, but may be seen as follows.
Since $\zeta+\zeta\inv\neq 0$, if $t\in\N$, $0\leq t\leq dr$, $t\equiv dr(\mod 2)$, then
for any monic diagram $u:t\to dr$, we have $u^*u=(\zeta+\zeta\inv)^{\frac{dr-t}{2}}\id_t\neq 0$, and hence if
$u$ is thought of as an element of $W_\zeta(t)$, $(u,u)\neq 0$. Hence for any such $t$, $L_\zeta(t)\neq 0$.
Although it is not needed for the proof of the theorem,
the fact that if $L_{\TL_{dr,\zeta}(\C)}(t)\neq 0$ then $\CF_d(L_{\TL_{dr,\zeta}(\C)}(t))\neq 0$ is verified in the same way, but requires a
computation, using the recurrence (*) in Example \ref{ex:sub} above, to show that for a non-zero element
$u=p_dD\in W_{\zeta}(t)$, where $D:t\to dr$ is a monic diagram, we have $(u,u)\neq 0$. That such elements exist is easily verified.

By the case $d=1$ of the Theorem, or, more precisely, \cite[Theorem 5.3]{GL98} applied to $\TL_{dr,\zeta}(\C)$, 
if $t\in\CN_2$, $0\leq t< g(t)\leq dr$ and $t\equiv dr(\mod 2)$, then $W_{\TL_{dr,\zeta}(\C)}(t)$
has composition factors $L_{\TL_{dr,\zeta}(\C)}(t)$ and $L_{\TL_{dr,\zeta}(\C)}(g(t))$. All other 
cell modules for $\TL_{dr,\zeta}(\C)$ are simple. It follows from the last paragraph that similarly,
if $t\in\CN_2$, $0\leq t< g(t)\leq dr$ and $t\equiv dr(\mod 2)$, then $W_{E_{\zeta}(d,r)}(t)$
has composition factors $L_{E_{\zeta}(d,r)}(t)$ and $L_{E_{\zeta}(d,r)}(g(t))$, and that other 
cell modules for ${E_{\zeta}(d,r)}$ are simple. All statements in the Theorem are now easy consequences of
standard cellular theory.
 \end{proof}

\begin{rem}
\begin{enumerate}
\item As a consequence of Theorem \ref{thm:dec} the equation \eqref{eq:dec} implies \eqref{eq:sum}, and the other statements in that sentence.
Thus the $\mu_\zeta(d,r;t)$ are determined by Theorem \ref{thm:dec}.
\item Since the $w(d,r;t)$ are known (Lemma \ref{lem:rec}(1)), it follows from Lemma \ref{lem:rec}(2) that the dimensions of
the weight spaces $\CT_\zeta(dr)_m$ are determined by Theorem \ref{thm:dec}.
\end{enumerate}
\end{rem}

\end{document}